\documentclass[reqno, 12pt]{amsart}   
\usepackage{graphicx}
\usepackage{amsmath,amssymb,amsthm, mathrsfs, mathtools, booktabs}
\usepackage[foot]{amsaddr}
\usepackage{xcolor}
\usepackage[backref]{hyperref}
\hypersetup{
    colorlinks,
    linkcolor={red!60!black},
    citecolor={green!60!black},
    urlcolor={blue!60!black}
}
\usepackage[T1]{fontenc}
\usepackage{lmodern}
\usepackage[babel]{microtype}
\usepackage[english]{babel}
\usepackage{comment}

\usepackage{geometry}
\numberwithin{equation}{section}
\numberwithin{figure}{section}

\usepackage{enumerate}
\usepackage{enumitem}

\theoremstyle{plain}
\newtheorem{thm}{Theorem}[section]
\newtheorem{theorem}[thm]{Theorem}

\newtheorem{conj}[thm]{Conjecture}
\newtheorem{cor}[thm]{Corollary}

\newtheorem{lemma}[thm]{Lemma}
\newtheorem{lem}[thm]{Lemma}

\DeclarePairedDelimiter{\parens}{(}{)}
\DeclarePairedDelimiter{\set}{\{}{\}}
\DeclarePairedDelimiter{\brackets}{[}{]}
\DeclarePairedDelimiter\size{\lvert}{\rvert}   

\newcommand{\Rodl}{R\"{o}dl}

\newcommand{\Turan}{Tur\'{a}n}

\newcommand{\Dudek}{Dudek}

\renewcommand{\leq}{\leqslant}
\renewcommand{\geq}{\geqslant}
\renewcommand\le{\leqslant}
\renewcommand\ge{\geqslant}

\DeclareMathOperator{\trace}{\mathsf{Tr}}
\DeclarePairedDelimiter{\floor}{\lfloor}{\rfloor}
\DeclarePairedDelimiter{\ceil}{\lceil}{\rceil}

\newcommand{\calP}{\ensuremath{\mathcal{P}}}
\newcommand{\calL}{\ensuremath{\mathcal{L}}}
\newcommand{\calI}{\ensuremath{\mathcal{I}}}

\title[The minimum degree of minimal Ramsey graphs for cliques]{The minimum degree of\\ minimal Ramsey graphs for cliques}

\author{John Bamberg$^1$}
\address[Bamberg]{$^1$Centre for the Mathematics of Symmetry and Computation, The University of Western Australia, Australia.}
\email[Bamberg]{john.bamberg@uwa.edu.au}

\author{Anurag Bishnoi$^2$}
\address[Bishnoi]{$^2$ Delft Institute of Applied Mathematics, TU Delft, Netherlands.}
\email[Bishnoi]{A.Bishnoi@tudelft.nl}
\thanks{Anurag Bishnoi's research supported by a Discovery Early Career Award of the Australian Research Council (No.~DE190100666)}

\author{Thomas Lesgourgues$^3$}
\address[Lesgourgues]{$^3$School of Mathematics and Statistics, UNSW, Australia.}
\email[Lesgourgues]{t.lesgourgues@unsw.edu.au}
\thanks{Thomas Lesgourgues' research supported by an Australian Government Research Training Program Scholarship and the School of Mathematics and Statistics, UNSW}
\thanks{}

\begin{document}

\begin{abstract}
    We prove that $s_r(K_k) = O(k^5 r^{5/2})$, where $s_r(K_k)$ is the Ramsey parameter introduced by Burr, Erd\H{o}s and Lov\'{a}sz 
    in 1976, which is defined as the smallest minimum degree of a graph $G$ such that any $r$-colouring of the edges of $G$ contains a monochromatic $K_k$, whereas no proper subgraph of $G$ has this property. 
    The construction used in our proof relies on a group theoretic model of generalised quadrangles introduced by Kantor in 1980. 
\end{abstract}

\maketitle 
\section{Introduction}
A graph $G$ is called $r$-Ramsey for another graph $H$, denoted by $G \rightarrow (H)_r$, if every $r$-colouring of the edges of $G$ contains a monochromatic copy of $H$. 
Observe that if $G \rightarrow (H)_r$, then every graph containing $G$ as a subgraph is also $r$-Ramsey for $H$. 
Some very interesting questions arise when we study graphs $G$ which are minimal with respect to $G \rightarrow (H)_r$, that is, $G \rightarrow (H)_r$ but there is no proper subgraph $G'$ of $G$ such that $G' \rightarrow (H)_r$. 
We call such graphs \textit{$r$-Ramsey minimal for $H$} and we denote the set of all $r$-Ramsey minimal graphs for $H$ by $\mathcal{M}_r(H)$. 
The classical result of Ramsey \cite{Ramsey:1929aa} implies that for any finite graph $H$ and positive integer $r$, there exists a graph $G$ that is $r$-Ramsey for $H$, that is, $\mathcal{M}_r(H)$  is non-empty.

Some of the  central problems in graph Ramsey theory are concerned with the case where $H$ is a clique $K_k$. 
For example, the most well studied parameter is the Ramsey number $R_r(k)$, that denotes the smallest number of vertices of any graph in $\mathcal{M}_r(K_k)$. 
The classical work of Erd\H{o}s \cite{Erdos:1947aa} and Erd\H{o}s and Szekeres \cite{Erdos:1935aa} shows that $2^{k/2} \le R_2(k) \le 2^{2k}$.
While these bounds have been improved since then, most recently by Sah \cite{Sah20} (also see \cite{Spencer:1975aa} and \cite{Conlon:2009aa}), the constants in the exponent have stayed the same. We refer the reader to the survey of Conlon, Fox and Sudakov \cite{Conlon:2015aa} for more on this and other graph Ramsey problems. 

Several other questions on $\mathcal{M}_r(H)$ have also been explored; for example, the well studied size-Ramsey number $\hat{R}_r(H)$ which is the minimum number of edges of a graph in $\mathcal{M}_r(H)$. 
We refer the reader to \cite{Burr81, Burr85, Luczak94, Rodl08} for various results on minimal Ramsey problems. In this paper, we will be interested in the \textit{smallest minimum degree of an $r$-Ramsey minimal graph}, which is defined by 
\[s_r(H) \coloneqq \min_{G \in \mathcal{M}_r(H)} \delta(G),\]
for a finite graph $H$ and positive integer $r$, where $\delta(G)$ denotes the minimum degree of $G$. 
Trivially, we have $s_r(H) \leq R_r(H) - 1$, since the complete graph on $R_r(H)$ vertices is $r$-Ramsey for $H$ and has minimum degree $R_r(H) - 1$. 
The study of this parameter was initiated by Burr, Erd\H{o}s and Lov\'{a}sz \cite{Burr:1976aa} in 1976.
They were able to show the rather surprising exact result, $s_2(K_k) = (k - 1)^2$, which is far away from the trivial exponential bound of $s_2(K_k) \leq R_r(k) - 1$. 
The behaviour of this function is still not so well understood for $r > 2$ colours. 
Fox et al.~\cite{Fox:2016aa} determined this function asymptotically for every fixed $k$ up-to a polylogarithmic factor, and for $k = 3$ their result was further improved by Guo and Warnke \cite{Guo-Warnke20} who managed to obtain matching logarithmic factors.

\begin{thm}[Fox, Grinshpun, Liebenau, Person, Szab\'{o}]\leavevmode
\label{thm:FGLPS1}
\begin{enumerate} 
    \item[(i)] There exist constants $c, C > 0$ such that for all $r \ge 2$, we have
    \[c r^2 \ln r \le s_r(K_3) \le C r^2 \ln^2 r.\]
    \item[(ii)] For all $k \ge 4$ there exist constants $c_k, C_k > 0$ such that for all $r \geq 3$, we have
    \[c_k r^2 \frac{\ln r}{ \ln{\ln r}} \le s_r(K_k) \le C_k r^2 (\ln r)^{8(k-1)^2}.\]
\end{enumerate}
\end{thm}

\begin{thm}[Guo, Warnke]
$s_r(K_3) = \Theta(r^2 \ln r)$. 
\end{thm}

The constant in the upper bound of Theorem \ref{thm:FGLPS1}(ii) is rather large ($C_k\sim k^2e^{4k^2\ln 2}$), and in particular not polynomial in $k$. 
To remedy this, they proved the following general upper bound which is polynomial in both $k$ and $r$.

\begin{thm}[Fox, Grinshpun, Liebenau, Person, Szab\'{o}]
\label{thm:FGLPS2}
For all $k, r \ge 3$, $s_r(K_k) \le 8(k - 1)^6 r^3$. 
\end{thm}

For a fixed $r$ and $k \rightarrow \infty$, H\`{a}n, R\"{o}dl and Szab\'{o} \cite{Han:2018aa} determined this function up-to polylogarithmic factors by proving the following.

\begin{thm}[H\`{a}n, R\"{o}dl, Szab\'{o}]
\label{thm:HRS}
There exists a constant $k_0$ such that for every $k > k_0$ and $r < k^2$
\[s_r(K_k) \leq 80^3 (r \ln r)^3 (k \ln k)^2.\]
\end{thm}

We prove the following general upper bound that improves Theorem \ref{thm:FGLPS2}, and thus provides the best known upper bound on $s_r(K_k)$ outside the special ranges covered by Theorem \ref{thm:FGLPS1} and \ref{thm:HRS}. 

\begin{thm}\label{thm:main}
There exists an absolute constant $C$ such that for all $r\geq 2, k \geq 3$, $s_r(K_k) \leq C (k-1)^5 r^{5/2}$. 
\end{thm}

Our proof uses the equivalence between $s_r(K_k)$ and another extremal function, called the \textit{$r$-colour $k$-clique packing number} \cite{Fox:2016aa},  defined as follows. Let $P_r(k)$ denote the minimum $n$ for which there exist $K_{k+1}$-free pairwise edge disjoint graphs $G_1, \dots, G_r$ on a common vertex set $V$ of size $n$ such that for any $r$-colouring of $V$, there exists an $i$ such that $G_i$ contains a copy of $K_k$ all of whose vertices are coloured in the $i$\textsuperscript{th} colour. 

\begin{lem}[{see \cite[Theorem 1.5]{Fox:2016aa}}]\label{lemma:min_equal_packing}
For all integers $r, k \ge 2$ we have $s_r(K_{k+1}) = P_r(k)$. 
\end{lem}

Our graphs $G_i$ in the packing would come from certain point-line geometries known as generalised quadrangles that we define in the next section. 
In Section \ref{sec:packing}, we show that any packing of `triangle-free' point-line geometries implies an upper bound on $P_r(k)$, assuming certain conditions on the parameters of the geometry.  
In Section, \ref{sec:construction} we give a packing of certain subgeometries of the so-called Hermitian generalised quadrangles using a group theoretic model given by Kantor in the 1980's \cite{Kantor:1986aa}, and deduce that this packing implies our main result.

\section{Background}\label{sec:background}

A (finite) generalised quadrangle $\mathcal{Q}$ of order $(s, t)$ is an incidence structure of points
$\mathcal{P}$, lines $\mathcal{L}$, together with a symmetric point-line
incidence relation satisfying the following axioms:
\begin{enumerate}[label=(\roman*)]
\item Each point lies on $t+1$ lines ($t\ge 1$) and two distinct points are incident with at most one line.\label{item:GQ_Axiom1}
\item Each line lies on $s+1$ points ($s\ge 1$) and two distinct lines are incident with at most one point.\label{item:GQ_Axiom2}
\item If $P$ is a point and $\ell$ is a line not incident with $P$, then there is a unique point on $\ell$ collinear with $P$.\label{item:GQ_Axiom3}
\end{enumerate} Notice that the third axiom above ensures that there are no triangles (i.e., three distinct lines pairwise meeting in three distinct points) in $\mathcal{Q}$.
The standard reference on finite generalised quadrangles is the book by Payne and Thas \cite{Payne:2009aa}.
The \textit{collinearity graph} of a generalised quadrangle is the graph on the set of points with two points adjacent when they are both incident with a common line. A collineation $\theta$ of $\mathcal{Q}$, that is, an automorphism of its collinearity graph, is an \textit{elation} about the point $P$ if it is either
the identity collineation, or it fixes each line incident with $P$ and fixes no point not collinear with $P$. If there is a group $E$ of elations of $\mathcal{Q}$ about the point $P$ such that $E$ acts regularly on the points not collinear with $P$, then we say that $\mathcal{Q}$ is an \textit{elation generalised quadrangle} with elation group $E$ and \textit{base point} $P$. Necessarily, $E$ has order $s^2t$, as there are $s^2t$ points not collinear to a given point in any generalised quadrangle.

Now suppose we have a finite group $E$ of order $s^2t$ where $s,t>1$. A \emph{Kantor family} of $E$ is a set $\mathcal{A}:=\{A_i\colon i=0,\ldots, t\}$ of subgroups of order $s$, and a set $\mathcal{A}^*:=\{A_i^*\colon i=0,\ldots, t\}$ of subgroups of order $st$, such that the following are satisfied:
\begin{enumerate}[label=(K\arabic*)]
    \setcounter{enumi}{-1}   
    \item $A_i\le A_i^*$ for all $i\in\{0,\ldots, t\}$; \label{item:AxiomK0}
    \item $A_i\cap A_j^*=\{1\}$ whenever $i\ne j$; \label{item:AxiomK1}
    \item $A_iA_j\cap A_k=\{1\}$ whenever $i,j,k$ are distinct. \label{item:AxiomK2}
\end{enumerate} 
Due to a theorem of Kantor (c.f., \cite[Theorem A.3.1]{Kantor:1986aa}), a Kantor family as described above, gives rise to an \emph{elation} generalised quadrangle of order $(s,t)$, which we briefly describe in Table \ref{table:PointsLines_ElationGQ}.
\begin{table}[ht]
    \begin{center}
    \begin{tabular}{l|l}
    \toprule
    \textsc{Points} & \textsc{Lines} \\
    \midrule
    elements $g$ of $E$&
    the right cosets $A_ig$\\
    right cosets $A_i^*g$&
    symbols $[A_i]$\\
    a symbol $\infty$.&\\
    \bottomrule
    \end{tabular}
    \end{center}
    
\begin{center}
\textsc{Incidence:}
\begin{tabular}{rcl}
$g$&$\sim $&$A_ig$\\
$A_i^*h$&$\sim$&$A_ig$, where $A_ig\subseteq A_i^*h$\\
$A_i^*h $&$\sim$&$[A_i]$\\
$\infty$&$\sim$&$[A_i]$
\end{tabular}
\end{center}
    \caption{The points and lines of the elation generalised quadrangle arising
      from a Kantor family (n.b., $A_i\in \mathcal{A}$, $A_i^*\in\mathcal{A}^*$, $g\in E$).}
    \label{table:PointsLines_ElationGQ}
\end{table}

We will simply be needing to use the Kantor family for a well-known family of generalised quadrangles, 
where the Heisenberg groups appear as the group $E$ in the description above. We remark that the main property
we will need is \ref{item:AxiomK2}, since it ensures that lines of the form $A_ig$, never form a triangle.For self-containment, we give a proof here.
Suppose $f,g,h$ are three elements of $E$ forming the vertices of a triangle. Then there are three elements $A,B,C\in\mathcal{A}$ such that $Af=Ag$, $Bg=Bh$, $Ch=Cf$. Therefore, $fg^{-1}\in A$, $gh^{-1}\in B$, $fh^{-1}\in C$, from which it follows that $fh^{-1}=(fg^{-1})(gh^{-1})\in AB\cap C$. Since $f\ne h$, we have $AB\cap C\ne \{1\}$.
So the condition $AB\cap C=\{1\}$ given by \ref{item:AxiomK2} ensures that there are no triangles.

In this paper, we will also need two commonly used maps on finite fields: the
\emph{trace} and \emph{norm} maps. For a finite field, they are the analogues of
`real part' and `square-modulus' for the complex numbers. We will not be needing the general theory of traces and norms, just two functions in particular. The \emph{relative trace map} from $\mathbb{F}_{q^2}$ to $\mathbb{F}_q$ is defined by $\trace(x) = x + x^q$.
If $\phi$ is a field-automorphism of $\mathbb{F}_{q^2}$, then
$\trace\left(\phi(x)\right)=\trace(x)$ for all $x\in\mathbb{F}_{q^2}$.
Also, $\trace$ is additive; that is, $\trace(x+y)=\trace(x)+\trace(y)$
for $\mathbb{F}_{q^2}$. The \emph{relative norm map} $\mathbb{F}_{q^2}$ to $\mathbb{F}_q$ is defined by $\mathsf{N}(x) = x^{q+1}$. It is also invariant under field-automorphisms, and is surjective. We will make use
of the following property: if $q$ is odd, then every element of $\mathbb{F}_q$
is a square in $\mathbb{F}_{q^2}$. To see this, let $t\in \mathbb{F}_q$. 
Since $\mathsf{N}$ is surjective, there exists $x\in \mathbb{F}_{q^2}$ such that 
$t=x^{q+1}$. Since $q$ is odd, the element $y=x^{(q+1)/2}$ is well defined, and then $t=y^2$.

\section{Packing Generalised Quadrangles}
\label{sec:packing}

A partial linear space is a point-line incidence structure with the property that any two distinct points are incident to at most one common line. 
A triangle-free partial linear space of order $(s, t)$ is an incidence structure satisfying Axioms \ref{item:GQ_Axiom1} and \ref{item:GQ_Axiom2} of a generalised quadrangle, and (iii)$'$ there are no three distinct lines pairwise meeting each other in three distinct points. 
Clearly, any subgeometry of a generalised quadrangle where the number of points on a line and the number of lines through a point are constants is a triangle-free partial linear space. 
We now prove the main lemma that will imply Theorem \ref{thm:main} once we have the construction outlined in Section \ref{sec:construction}. 
Our proof follows the same idea as in \Dudek{} and \Rodl{} \cite{Dudek:2011aa}, and Fox et al. \cite{Fox:2016aa}.\bigskip

\begin{lemma}\label{lemma:PackingNumber}
Let $r, k, s, t$ be positive integers. 
Say there exists a family $(\calI_i)_{i=1}^r$ of triangle-free partial linear spaces of order $(s, t)$, on the same point set $\calP$ and pairwise disjoint line-sets $\calL_1, \dots, \calL_r$, such that the point-line geometry $\parens*{\calP, \bigcup_{i=1}^r \calL_i}$ is also a partial linear space. If $s \geq 3rk\ln k$ and $t  \geq 3k(1+\ln r)$, then $P_r(k) \leq \size{\calP}$. 
\end{lemma}

\begin{proof}
In order to show that $P_r(k)\leq \size{\calP}$, we will exhibit $K_{k+1}$-free pairwise edge disjoint graphs $G_1, \dots, G_r$ on the common vertex set $V=\calP$, such that for any $r$-colouring of $V$, there exists an $i$ such that $G_i$ contains a copy of $K_k$ all of whose vertices are coloured in the $i$\textsuperscript{th} colour.
We start by recalling the following properties about each partial linear space $\calI_i$, $i\in \set{1,\ldots,r}$:
\begin{enumerate}[label=(P\arabic*)]
    \item Every point $p\in\calP$ is incident with $t+1$ lines of $\calL_i$.\label{item:PLS_LinesPerPoint}
    \item Every line $\ell\in\calL_i$ contains $s+1$ points from $\calP$.\label{item:PLS_PointsPerLine}
    \item Any two points of $\calP$ lie on at most one line of $\calL_i$.\label{item:PLS_TwoPointsOneLine}
    \item $\calI_i$ is triangle-free. \label{item:PLS_TriangleFree}
\end{enumerate}
Furthermore, given that $\parens*{\calP, \bigcup_{i=1}^r \calL_i}$ is a partial linear space and the line-sets $\calL_1, \dots, \calL_r$ are disjoint, 
\begin{enumerate}[label=(P\arabic*),resume]
    \item For any $i\neq j$, and any $\ell \in \calL_i$, $m \in \calL_j$, $\ell$ and $m$ are incident with at most one common point.\label{item:PLS_EdgeDisjoint}
\end{enumerate}

Let $i\in\set{1,\ldots,r}$, $\ell_1=\floor*{\frac{s+1}{k}}$ and $\ell_2=\ceil*{\frac{s+1}{k}}$. For each line $\ell\in\calL_i$, we select uniformly at random one partition of $\ell$ among all  $\ell=\bigcup_{j=1}^k L_j^{(\ell)}$, where $L_j^{(\ell)}$ denotes the $j$\textsuperscript{th} component of the partition, such that for some $k'$, $\size{L_1^{(\ell)}} , \ldots , \size{L_{k'}^{(\ell)}} = \ell_1$ and  $\size{L_{k'+1}^{(\ell)}} , \ldots , \size{L_{k}^{(\ell)}} = \ell_2$. Choices for distinct lines in $\calL_i$ are independent. \bigskip

We define a graph $G_i=(V,E_i)$ on the vertex set $V =\calP$ as follows. For every $\ell\in\calL_i$, we include the edges of a complete $k$-partite graph between the vertex sets $L_{j}^{(\ell)}$ for $j\in \set{1,\ldots,k}$. Note that the graph $G_i$ is a collection of \Turan{} graphs on $(s+1)$ vertices with $k$ parts. Each \Turan{} graph comes from one line $\ell\in\calL_i$. By property \ref{item:PLS_TwoPointsOneLine}, any two points are incident with at most one line, therefore the different \Turan{} graphs are edge-disjoint. Furthermore, by property \ref{item:PLS_TriangleFree}, $G_i$ is $K_{k+1}$-free. Finally, by property \ref{item:PLS_EdgeDisjoint}, for any $i\neq j \in \set{1,\ldots,r}$, $G_i$ and $G_j$ are edge disjoint.

In order to conclude, we need to show that with positive probability, for any $r$-colouring of $V$, there exists an $i$ such that $G_i$ contains a copy of $K_k$ all of whose vertices are coloured in the $i$\textsuperscript{th} colour. Note that given $G_1,\ldots,G_r$ on the vertex set $V=\calP$, in any $r$-colouring of $V$, at least one of the colours occurs at least $\size{\calP}/r$ times. Therefore if for every $G_i$, every set of at least $\size{\calP}/r$ vertices contains a copy of $K_k$, then we get the desired property. The choices of partitions being done independently, to conclude our proof it suffices to show that for each $i\in\set{1,\ldots,r}$, with positive probability every set of at least $\size{\calP}/r$ vertices contains a copy of $K_k$ in $G_i$.\bigskip

Fix $i\in\set{1,\ldots,r}$. For a subset $W\subseteq \calP$, let $\mathcal{A}(W)$ denote the event that the induced subgraph $G_i[W]$ contains no copy of $K_k$. Let $U\subset \calP$ with $\size{U}=\floor*{\frac{\size{\calP}}{r}}$. By property \ref{item:PLS_TriangleFree}, any copy of $K_k$ can only appear from one line $\ell\in\calL_i$, i.e.
\[\mathcal{A}(U)\subseteq\bigcap_{\ell\in\calL_i}\mathcal{A}(U\cap\ell).\]

All the events $\mathcal{A}(U\cap\ell)$ are independent, therefore
\[\mathbb{P}(\mathcal{A}(U))\leq\prod_{\ell\in\calL_i}\mathbb{P}(\mathcal{A}(U\cap\ell)).\]

For a given line $\ell\in\calL_i$, let $u_\ell= \size{U\cap\ell}$, and let $\ell=\bigcup_{j=1}^k L_j^{(\ell)}$ be the random partition of $\ell$. Note that $U\cap\ell$ contains no copy of $K_k$ if and only if there exists $j \in \set{1,\ldots,k}$ such that $U\cap L_j^{(\ell)}=\varnothing$. For a fixed $j\in\set{1,\ldots,k}$,

\[ \mathbb{P}\left( U\cap L_j^{(\ell)}=\varnothing\right) =
\frac{\binom{s+1-u_\ell}{\vert L_j^{(\ell)} \vert}}{\binom{s+1}{\vert L_j^{(\ell)} \vert}}
\leq \left( 1 - \frac{u_\ell}{s+1}\right)^{\vert L_j^{(\ell)}\vert}
\leq \exp\left(-\frac{\ell_1u_\ell}{s+1}\right).
\]

Therefore
\begin{align*}
    \mathbb{P}(\mathcal{A}(U))
    &\leq \prod_{\ell\in \calL_i}\mathbb{P}\left( \exists \ j\in\set{1,\ldots,k}, \ U\cap L_j^{(\ell)}=\varnothing\right)\\
    &\leq k^{\size{\calL_i}}\exp\left(-\sum_{\ell\in\calL_i}\frac{\ell_1u_\ell}{s+1}\right). 
\end{align*}

Because every point of $U$ is incident with $t+1$ lines from $\calL_i$ (by property \ref{item:PLS_LinesPerPoint}), $\sum_{\ell\in\calL_i}u_\ell = \sum_{\ell\in\calL_i}\size{U\cap\ell} = (t+1)\size{U}$, and thus
\[\mathbb{P}(\mathcal{A}(U))\leq k^{\size{\calL_i}}\exp\left(-\frac{\ell_1(t+1)\size{U}}{s+1}\right).\]
Finally,

\begin{align*}
    \mathbb{P}\left(\exists U \in \binom{\calP}{\floor*{ \frac{\size{\calP}}{r}}} :\ \mathcal{A}(U)\right) 
    &\leq \binom{\size{\calP}}{\floor*{ \frac{\size{\calP}}{r}}}k^{\size{\calL_i}}\exp\left(-\frac{t+1}{s+1}\ell_1\floor*{\frac{\size{\calP}}{r}}\right)\\
    &\leq (re)^{\size{\calP}/r}k^{\size{\calL_i}}\exp\left(-\frac{t+1}{s+1}\floor*{\frac{s+1}{k}}\floor*{\frac{\size{\calP}}{r}}\right)
\end{align*}
Given that $\size{\mathcal{P}}\geq s\geq 3rk\ln k$, $r\geq2$ and $k\geq3$, we have
\[ \floor*{\frac{s+1}{k}} \geq \frac{5}{6}\frac{s+1}{k},\qquad \floor*{\frac{\size{\calP}}{r}} \geq\frac{5}{6}\frac{\size{\calP}}{r}.\]

Therefore,
\begin{align*}
    \mathbb{P}\left(\exists U \in \binom{\calP}{\floor*{ \frac{\size{\calP}}{r}}} :\ \mathcal{A}(U)\right) 
    &\leq (re)^{\size{\calP}/r}k^{\size{\calL_i}}\exp\left(-\frac{t+1}{s+1}\frac{5(s+1)}{6k}\frac{5\size{\calP}}{6r}\right)\\
    &\leq \exp\brackets*{\size{\calP}\parens*{
    \frac{1+\ln r}{r}+\frac{\size{\calL_i}}{\size{\calP}}\ln{k}-\frac{25}{36}\frac{t+1}{rk}}}.
\end{align*}

By double counting (using properties \ref{item:PLS_LinesPerPoint} and \ref{item:PLS_PointsPerLine}) we know that
\[\size{\calL_i}(s+1)=\size{\calP}(t+1),\]

and therefore
\begin{equation}
\mathbb{P}\left(\exists U \in \binom{\calP}{\floor*{ \frac{\size{\calP}}{r}}} :\ \mathcal{A}(U)\right) 
\leq
\exp\brackets*{\size{\calP}\parens*{
\frac{1+\ln r}{r}+\frac{t+1}{s+1}\ln k-\frac{25}{36}\frac{t+1}{rk}}} \label{eq:optim}
\end{equation}

Note that since $s\geq 3rk\ln k$ we have \[\frac{25}{36}\frac{t+1}{rk}> \frac{2(t+1)}{s+1}\ln k,\] and since $t\geq 3k(1+\ln r)$ we have
\[\frac{25}{36}\frac{t+1}{rk}> \frac{2(1+\ln r)}{r}.\]
Therefore,
\[\mathbb{P}\left(\exists U \in \binom{\calP}{\floor*{ \frac{\size{\calP}}{r}}} :\ \mathcal{A}(U)\right) <1.\]

Then there exists an instance of $G_i$ such that every subset of $\calP$ with at least $\floor*{\frac{\size{\calP}}{r}}$ vertices contains a copy of $K_k$ in $G_i$.
\end{proof}

\section{The construction}
\label{sec:construction}

Let $q$ be a prime power, and denote the finite field of order $q^2$ by $\mathbb{F}_{q^2}$.
We will use the model of the Hermitian generalised quadrangle $H(3,q^2)$ that appears in \cite[Section 3]{Kantor:1980} (see Example 3). 
For a definition of $H(3, q^2)$ see \cite[Chapter 3]{Payne:2009aa}.

Let $E$ be the group defined on $\mathbb{F}_{q^2} \times \mathbb{F}_q \times \mathbb{F}_{q^2}$ by the following operation:
$$(a, \gamma, b) \circ (a', \gamma', b') = (a + a', \gamma  + \gamma' + \trace(b^qa'), b + b'),$$ where $\trace(x) = x + x^q$ is the relative trace map from $\mathbb{F}_{q^2}$ to $\mathbb{F}_q$.
It turns out that $E$ is the \emph{Heisenberg group} of order $q^5$ with centre of order $q$. 

We can construct a generalised quadrangle by constructing a Kantor family of $E$. Define
\begin{align*}
A^*_\infty&=\{ (0,\gamma,a) \colon a\in \mathbb{F}_{q^2}, \gamma\in \mathbb{F}_{q}\},\\
A^*_t&=\{ (a,\gamma,at) \colon a\in \mathbb{F}_{q^2}, \gamma\in \mathbb{F}_{q}\},\quad t\in \mathbb{F}_{q},\\
A_\infty&=\{ (0,0,a) \colon a\in \mathbb{F}_{q^2}\},\\
A_t&=\{ (a,a^{q+1}t,at) \colon a\in \mathbb{F}_{q^2}\},\quad t\in \mathbb{F}_{q}.
\end{align*}
Then $\mathcal{A}:=\{A_\infty\}\cup \{A_t\colon t\in \mathbb{F}_{q}\}$ and
$\mathcal{A}^*:=\{A^*_\infty\}\cup \{A^*_b\colon b\in \mathbb{F}_{q}\}$
form a Kantor family of $E$ giving rise to a generalised quadrangle isomorphic to $H(3,q^2)$.\footnote{In \cite{Kantor:1980} the dual of this generalised quadrangle is defined, denoted by $O^{-}(6, q)$, but it is well known that $H(3, q^2)$ is isomorphic to the dual of the elliptic generalised quadrangle $O^{-}(6, q)$ (see \cite[Chapter 3]{Payne:2009aa}, where $O^{-}(6, q)$ is denoted by $Q(5,q)$).}  


From now on we will assume that $q$ is odd.
Let $\kappa$ be an element of $\mathbb{F}_{q^2}$. 
For each $\lambda\in\mathbb{F}_{q^2}$, define $\tau_\lambda : E \rightarrow E$ as follows.
Let 
\begin{align*}
\tau_\lambda\colon& (a,0,0)\mapsto\left(a,\trace( \lambda a+ \kappa \lambda a^q +\tfrac{1}{2}\lambda^q a^2),  \lambda a^q\right),\\
\tau_\lambda\colon& (0,\gamma,b)\mapsto(0,\gamma,b).
\end{align*}
By Axiom \ref{item:AxiomK1}, we have $|A_0 A_\infty^*|=|A_0||A_\infty^*|/
|A_0\cap A_\infty^*|=s(st)/1=|E|$. So $E=A_0 A_\infty^*$.
Therefore, we can write every element $g \in E$ as $g = g_0 g_\infty^*$, with $g_0 \in A_0$ and $g_\infty^* \in A_\infty^*$. 
Define $\tau_\lambda(g) \coloneqq \tau_\lambda(g_0) \circ \tau_\lambda(g_\infty^*)$. 

\begin{lemma}\label{lem:aut}
For every $\lambda \in \mathbb{F}_{q^2}$, $\tau_\lambda$ is an automorphism of $E$.
\end{lemma}

\begin{proof}
Let $\lambda\in\mathbb{F}_{q^2}$.
It suffices to show that $\tau_\lambda$ is a homomorphism from $A_0$ to $E$, since $\tau_\lambda$
is clearly bijective. Let $a_1,a_2\in \mathbb{F}_{q^2}$. Then
\[\tau_\lambda(a_1 + a_2, 0, 0)=\left(a_1+a_2,\trace( \lambda (a_1+a_2)+\kappa\lambda (a_1+a_2)^q +\tfrac{1}{2}\lambda^q (a_1+a_2)^2),
\lambda (a_1+a_2)^q\right)\]
Using $(x + y)^q = x^q + y^q$ in $\mathbb{F}_{q^2}$, we get,
\begin{align*}
    \tau_\lambda(a_1 + a_2, 0, 0)=& \left( a_1 + a_2, \trace(\lambda a_1 + \kappa \lambda a_1^q + \lambda^q a_1^2 + \lambda a_2 + \kappa \lambda a_2^q + \lambda^q a_2^2 + \lambda^q a_1a_2) , \lambda a_1^q+a_2^q\right).
\end{align*}
Using the fact that $\trace{(\cdot)}$ is additive, we obtain,
\begin{align*}
\tau_\lambda(a_1 + a_2, 0, 0)=&\big(a_1+a_2,\trace( \lambda a_1+\kappa\lambda a_1^q+\tfrac{1}{2}\lambda^q a_1^2)+\\
&\trace(\lambda a_2+\kappa\lambda a_2^q+\tfrac{1}{2}\lambda^q a_2^2)+\trace(\lambda^q a_1a_2),\lambda a_1^q+\lambda a_2^q\big)\\
=&\left(a_1,\trace( \lambda a_1+\kappa\lambda a_1^q +\tfrac{1}{2}\lambda^q a_1^2),  \lambda a_1^q\right)\circ
\left(a_2,\trace( \lambda a_2+\kappa\lambda a_2^q +\tfrac{1}{2}\lambda^q a_2^2),  \lambda a_2^q\right)\\
=&\tau_\lambda(a_1,0,0) \circ \tau_\lambda (a_2, 0, 0).
\end{align*}
Therefore, $\tau_\lambda$ is an automorphism of $E$.
\end{proof}

\begin{lemma}\label{lemma:existsk}
For every odd prime power $q$, there exists
$\kappa\in\mathbb{F}_{q^2}$ such that $\trace(\kappa a+a^{2q-1})\ne 0$ for all nonzero $a\in \mathbb{F}_{q^2}$.
\end{lemma}

\begin{proof}
Suppose that $\trace(\kappa a + a^{2q - 1}) = \kappa a + \kappa^q a^q + a^{2q - 1} + a^{2 - q} = 0$ for some non-zero $a\in \mathbb{F}_{q^2}$. Multiplying by $a^{q - 2}$ and defining $y = -a^{q - 1}$ we get the following cubic equation,
\[
y^3 - \kappa^q y^2 + \kappa y - 1 = 0.\]
Therefore if $\kappa$ is such that the cubic is irreducible, then we get a contradiction. 
We show that such a choice of $\kappa$ exists. 
Let $t$ be a non-square in $\mathbb{F}_{q}$ and suppose that $\alpha$ is an element of $\mathbb{F}_{q^2}$ such
that $\alpha^2=t$.
By the Hansen-Mullen Irreducibility Conjecture (which is true, see \cite[Theorem 2.7]{Cohen:2005aa})
there exists a monic cubic of the form $x^3+ux^2-tx+v$ irreducible in $\mathbb{F}_q[x]$ for some $u, v \in \mathbb{F}_q$.
Let
\[
\kappa=\left(\frac{-tu+3v}{tu+v}+\frac{4t}{tu+v}\alpha\right)^q.
\]
By \cite[Theorem 3]{Kim:2009aa}, $y^3 - \kappa^q y^2 + \kappa y - 1$ is irreducible in $\mathbb{F}_{q^2}[y]$.
\end{proof}

\begin{theorem}\label{theorem:construction}\samepage
Let $\kappa$ be an element of $\mathbb{F}_{q^2}$ such that $\trace(\kappa a+a^{2q-1})\ne 0$ for all non-zero $a\in \mathbb{F}_{q^2}$.
 For each $\lambda\in\mathbb{F}_{q^2}$ and $t\in\mathbb{F}_q$, let
\[
A_t^\lambda = \{(a, a^{q+1}t + \trace(\lambda a +\kappa\lambda a^q+\tfrac{1}{2}\lambda^q a^2), at +\lambda a^q) : a \in \mathbb{F}_{q^2}\},
\]
and let 
\[
\mathcal{S}:=\{ \{A_t^\lambda\colon t\in\mathbb{F}_q\}\colon \lambda\in\mathbb{F}_{q^2}\}.
\]
Then:
\begin{enumerate}[label=(\roman*), ref=\ref{theorem:construction}(\roman*)]
\item Every element of $\mathcal{S}$ is a set of subgroups and any two cosets from subgroups of different such sets intersect each other in at most one element.\label{theorem:construction_item_i}
\item If we let $\mathcal{P}$ be the underlying set of $E$, then for every $\lambda \in \mathbb{F}_{q^2}$ the set of lines $\mathcal{L}_\lambda = \{A_t^\lambda g : g \in E, t \in \mathbb{F}_q\}$ gives rise to a triangle-free partial linear space $(\mathcal{P}, \mathcal{L}_\lambda)$ of order $(q^2-1,q-1)$. \label{theorem:construction_item_ii}
\end{enumerate}
\end{theorem}

\begin{proof}
First, for each $\lambda\in\mathbb{F}_{q^2}$ and $t\in\mathbb{F}_q$, we have 
\[
A_t^\lambda=\tau_\lambda(A_t).
\]
Since $\tau_\lambda$ is an automorphism of the underlying group (as per Lemma \ref{lem:aut}), it follows that $\{\tau_\lambda(A_\infty)\}\cup \{\tau_\lambda(A_t)\colon t\in \mathbb{F}_{q}\}$ and
$\{\tau_\lambda(A^*_\infty)\}\cup \{\tau_\lambda(A^*_b)\colon b\in \mathbb{F}_{q}\}$ also form a Kantor family, and give rise to an isomorphic generalised quadrangle. 
Therefore, by taking cosets of the subgroups $A_t^\lambda$, $t \in \mathbb{F}_q$, as lines we get a subgeometry of this generalised quadrangle which is a triangle-free partial linear space of order $(q^2 - 1, q - 1)$. 
We now prove the first part. 

First we show that $A_{t_1}^{\lambda_1}\cap A_{t_2}^{\lambda_2}=\{(0,0,0)\}$ whenever $\lambda_1\ne \lambda_2$.
An element of $A_{t_1}^{\lambda_1}\cap A_{t_2}^{\lambda_2}$ is of the form
$(a, a^{q+1}t_1 + \trace(\lambda_1 a +\kappa\lambda_1 a^q+\tfrac{1}{2}\lambda_1^q a^2), at_1 +\lambda_1 a^q)$
for some $a\in\mathbb{F}_{q^2}$, but it is also
$(b, b^{q+1}t_2 + \trace(\lambda_2 b +\kappa\lambda_2 b^q+\tfrac{1}{2}\lambda_2^q b^2), bt_2 +\lambda_2 b^q)$
for some $b\in\mathbb{F}_{q^2}$. Therefore, $a=b$ and hence
\begin{align*}
a^{q+1}(t_1-t_2)&=\trace\left((\lambda_2-\lambda_1)( a +\kappa a^q+\tfrac{1}{2} a^{2q})\right)\\
a(t_1-t_2)&=(\lambda_2-\lambda_1) a^q.
\end{align*}
Now $a^{q+1}(t_1-t_2)=a^qa(t_1-t_2)$ and so
\begin{align*}
\trace\left((\lambda_2-\lambda_1)( a +\kappa a^q+\tfrac{1}{2} a^{2q})\right)
&=a^{2q}(\lambda_2-\lambda_1).
\end{align*}
Expanding this equation gives us
\[
(\lambda_2-\lambda_1)( a +\kappa a^q-\tfrac{1}{2} a^{2q})+(\lambda_2-\lambda_1)^q( a +\kappa a^q+\tfrac{1}{2} a^{2q})^q=0.
\]
Suppose, by way of contradiction, that $a\ne 0$.
Since $\lambda_2 \neq \lambda_1$, we can rewrite the equation as
\[
(\lambda_2-\lambda_1)^{q-1}=-\frac{a +\kappa a^q-\tfrac{1}{2} a^{2q}}{( a +\kappa a^q+\tfrac{1}{2} a^{2q})^q}.
\]
Let $\mathsf{N}\colon \mathbb{F}_{q^2} \rightarrow \mathbb{F}_q$ be the relative norm function, defined by $\mathsf{N}(x) = x^{q+1}$. 
The left-hand side has norm 1 and hence
\[
\mathsf{N}(-(a^q +\kappa^qa-\tfrac{1}{2} a^{2}))=\mathsf{N}( a^q +\kappa^qa+\tfrac{1}{2} a^{2}).
\]
Now $\mathsf{N}(-1)=1$ and we can factor out $\mathsf{N}(a)$, so
\[
\mathsf{N}( a^{q-1}+\kappa^q+\tfrac{1}{2} a)-\mathsf{N}(a^{q-1} +\kappa^q-\tfrac{1}{2} a)=0.
\]
Note that by definition of the relative norm map, 
\begin{align*}
\mathsf{N}(x+y+z)
&=(x+y+z)(x+y+z)^q\\
&=(x+y+z)(x^q+y^q+z^q)\\
&= xx^q+yy^q+zz^q + xy^q+yx^q + xz^q+zx^q + yz^q+zy^q
\end{align*}
Using that in $\mathbb{F}_{q^2}$, $yx^q = (y^qx)^q$, we obtain
\[ \mathsf{N}(x+y+z) = \mathsf{N}(x)+\mathsf{N}(y)+\mathsf{N}(z) + \trace{(xy^q+xz^q+y^qz)}.\]

Therefore,
\begin{align*}
0=&\left(\mathsf{N}(a^{q-1}) +\mathsf{N}(\kappa^q)+\mathsf{N}(\tfrac{1}{2} a)+\trace(a^{q-1}\kappa+\tfrac{1}{2}a^{2q-1}+\tfrac{1}{2}\kappa a)\right)\\
&-\left(\mathsf{N}(a^{q-1}) +\mathsf{N}(\kappa^q)+\mathsf{N}(\tfrac{1}{2} a)+\trace(a^{q-1}\kappa-\tfrac{1}{2}a^{2q-1}-\tfrac{1}{2}\kappa a)\right)\\
=&\trace(\kappa a +a^{2q-1}),
\end{align*}
a contradiction. So $a=0$ and $A_{t_1}^{\lambda_1}\cap A_{t_2}^{\lambda_2}=\{(0,0,0)\}$.

Now for any two subgroups $H, K$ of a group, the intersection of two cosets of $H$ and $K$ is either empty, or a coset of $H \cap K$, which proves our claim. 
\end{proof}

\begin{cor}
There exists an absolute constant $C$ such that for all $r\geq 2, k \geq 3$, we have $s_r(K_k) \leq C (k-1)^5 r^{5/2}$.
\end{cor}
\begin{proof}
Let $r\geq 2$, $k\geq 3$, $c=\frac{10+9\ln{2}}{3\sqrt{2}}$ and let $q$ be the smallest prime such that $q\geq ck\sqrt{r}$. By Lemma \ref{lemma:existsk} and Theorem \ref{theorem:construction_item_ii}, there exists a family of $r\leq q^2$ triangle-free partial linear spaces of order $(q^2-1, q-1)$, on the same point set $\calP$ and pairwise disjoint line-sets $\calL_1, \dots, \calL_r$, and by Theorem \ref{theorem:construction_item_i}, the point-line geometry $\parens*{\calP, \bigcup_{i=1}^r \calL_i}$ is also a partial linear space. Note that $q^2-1 \geq 3rk\ln k$ and $q -1 \geq 3k(1+\ln r)$. Combining Lemmas \ref{lemma:min_equal_packing} and \ref{lemma:PackingNumber}, $s_r(K_{k+1}) = P_r(k)\leq \size{\calP}$. By Bertrand's postulate, $q\leq 2ck\sqrt{r}$, and using $\size{\calP}= q^5$ yields the desired bound, with $C=\ceil{(2c)^5}=26,\!282$. Note that, in light of Conjecture \ref{conj:final_conj}, we did not try to further optimise this constant.
\end{proof}

\section{Concluding remarks}
While generalised quadrangles have been used extensively in extremal combinatorics, and particularly Ramsey theory (e.g. \cite{Dudek:2011aa,furedi_turan_2006,Kostochka:2013,Mubayi:2019,tait_degree_2018}), our result appears to be the first instance in Ramsey theory where the group theoretic structure of these geometries is exploited. 
We are hopeful that Kantor's model of generalised quadrangles will lead to new results in other Ramsey problems as well.

In the probabilistic argument of Section \ref{sec:packing}, note that if we use $s + 1 =  q^2$ and $t + 1 =  q$, then from equation \eqref{eq:optim} it follows that we can solve the following quadratic inequality in $q$ to ensure that the probability is $< 1$: 
\[\frac{25}{36}\frac{1}{rk}q^2-\frac{1+\ln r}{r}q-\ln k > 0.\]
One can check that this inequality is satisfied for all $q \geq \frac{36}{25}k(1+\ln r) + \frac{6}{5}\sqrt{rk\ln k}$. Using that for any $a,b>0$, $(a+b)^5\leq 2^4(a^5+b^5)$, we obtained the following more refined upper bound. 
\begin{thm}
For all $r\geq 2, k \geq 2$, \[s_r(K_{k}) \leq 2^{12}\brackets*{ (k-1)^5\ln^5 r + (k-1)^{5/2}r^{5/2}\ln^{5/2}(k-1)}.\]
\end{thm}

For further improvements to our upper bound we should perhaps explore triangle-free partial linear spaces that do not arise from generalised quadrangles. Moreover, if we could make the probabilistic argument of Section \ref{sec:packing} deterministic, then this could also lead to an improvement in the bound. 
We would like to make the following conjecture.\\

\begin{conj} For all $r \geq 2, k \geq 2$ \[s_r(K_k) \leq C k^2 r^2 f(\ln k,\ln r)\] for some constant $C > 0$ and a constant degree polynomial function $f$.\label{conj:final_conj}
\end{conj}


\subsection*{Acknowledgements} We are grateful to Anita Liebenau for helpful discussions. We also thank Simona Boyadzhiyska for pointing out a significant gap in our first draft that was fixed later. We also want to thank the anonymous referee for a thorough proofreading and many helpful remarks.

\bibliographystyle{abbrv}
\bibliography{references}
\end{document}